\documentclass[12pt,reqno]{amsart}
\usepackage[skip=5pt plus1pt, indent=20pt]{parskip}
\usepackage{palatino}
\usepackage{amsmath,amssymb,amsxtra,color,calligra,mathrsfs,tcolorbox}

\usepackage{xcolor} % Required for specifying custom colours
\colorlet{mdtRed}{red!50!black}
\colorlet{dblue}{blue!50!black}
\usepackage[colorlinks,pagebackref=true]{hyperref}
\hypersetup{linkcolor=dblue,citecolor=dblue,filecolor=dullmagenta,urlcolor=mdtRed}
\renewcommand*{\backref}[1]{}
\renewcommand*{\backrefalt}[4]{[{%
		\ifcase #1 Not cited.%
		\or $\uparrow$~#2.%
		\else $\uparrow$~#2.%
		\fi%
	}]}
\usepackage[all]{xy}
\usepackage{tikz,tikz-cd,tkz-graph,enumerate}
\usetikzlibrary{matrix,arrows,decorations.pathmorphing}

\usepackage[us,12hr]{datetime}
\usepackage{comment}

\DeclareMathOperator{\Pic}{\textnormal{Pic}}

\DeclareMathOperator{\Spec}{{\rm Spec}}
\DeclareMathOperator{\Id}{{\rm Id}}

\DeclareMathOperator{\dv}{\rm div}

\newcommand{\mc}[1]{\mathcal{#1}}

\newcommand{\bb}[1]{\mathbb{#1}}

\newcommand*{\homsheaf}{\mathcal{H}\,\,\kern -5pt {\large om}}
\newcommand{\parhomsheaf}{\mathcal{PH}\,\,\kern -5pt {\large om}}

\newcommand{\doi}[1]{\href{https://doi.org/#1}{doi:#1}}
\numberwithin{equation}{section}

\newtheorem{theorem}[equation]{Theorem}

\newtheorem{lemma}[equation]{Lemma}
\newtheorem{proposition}[equation]{Proposition}

\theoremstyle{definition}
\newtheorem{definition}[equation]{Definition}
\newtheorem{remark}[equation]{Remark}
\newtheorem{example}[equation]{Example}

\makeatletter
\newcommand\fnsymb[1]{\textsuperscript{\@fnsymbol{#1}}}
\newcommand\fnletter[1]{\lowercase{\textsuperscript{\@alph{#1}}}}
\newcommand\fnnum[1]{\textsuperscript{#1}}
\makeatother

\usepackage{marvosym} % For Email/Letter icons. 
\renewcommand{\email}[2][1]{\thanks{\textit{Email address}#1: \href{mailto:#2}{#2}}}

\renewcommand{\address}[2][1]{\thanks{\textit{Address}#1: #2}} 

\begin{document}
	\baselineskip=15.5pt 
	
	\title[Artin-Schreier Root Stacks and lifts of group actions]{Artin-Schreier Root Stacks and lifts of group actions}
	
	\author[S. Chakraborty]{Sujoy Chakraborty\fnnum{1}}
	\address[\fnnum{1}]{Department of Mathematics, 
		Indian Institute of Science Education and Research Tirupati, Andhra Pradesh - 517619, India.}
	\email[\fnnum{1}]{sujoy.cmi@gmail.com}
	\author[S. Majumder]{Souradeep Majumder\fnnum{2}}
	\address[\fnnum{2}]{Department of Mathematics, 
		Indian Institute of Science Education and Research Tirupati, Andhra Pradesh - 517619, India.} 
	\email[\fnnum{2}]{souradeep@iisertirupati.ac.in}
	
	%]	
	%	\thanks{Corresponding author: Sujoy Chakraborty}
	
	\subjclass[2020]{14A20, 14D23, 14L30.}
	
	\keywords{Stack, Wild ramification, Artin-Schreier extension, Root stack.}
	
	\begin{abstract}
		Let $G$ be a connected affine algebraic group defined over a field of positive characteristic. We prove that the action of $G$ on a smooth projective variety can be lifted to its associated Artin-Schreier root stacks, whenever $G$ has no non-trivial characters. The existence of a $G$-linearization on a certain tautological invertible sheaf on such Artin-Schreier root stacks is also shown. 
	\end{abstract}
	
	%	\baselineskip=15.5pt 
	%	
	%	\title[Orthogonal and symplectic parabolic bundles and stack of roots]{Orthogonal and symplectic parabolic bundles and root stacks}
	%	
	%	\subjclass[2010]{14D23, 14H60, 53B15, 53C05}
	%	\keywords{Parabolic bundle; Root stack} 	
	
	\maketitle
	\section{Introduction}
	Over an algebraically closed field of characteristic 0, any smooth orbifold curve is isomorphic to a root stack over its coarse moduli space. This is essentially due to the fact that, the stabilizer groups of the stacky points are always cyclic. Due to this nice property, a smooth orbifold curve is uniquely determined upto isomorphism by its coarse moduli curve and a finite list of integers corresponding to the orders of the (cyclic) stabilizer groups of the stacky points. In positive characteristic, such nice description fails due to the possibility of non-cyclic -- or even non-abelian -- stabilizer groups. This necessitates the study of finer invariants like ramification jumps in positive characteristic. Ramified covers of curves in characteristic $p>0$ can be studied using various ramification filtrations of their Galois groups. For example, given a $\bb{Z}/p{\bb{Z}}$-Galois cover of curves $Y\rightarrow X$ over an algebraically closed field $K$ of characteristic $p$,  taking the completions of the function fields $K(Y)$ and $K(X)$, one gets a Galois extension of degree $p$ of the form $K((y))/K((x))$, where $y$ satisfies an equation of the form $y^p-y=f(x)$. In this case, the ramification jump is $m=-v(f)$, where $v$ is the usual discrete valuation associated to $K((x))$.  
	The notion of Artin-Schreier root stacks was introduced by Kobin in \cite{K21} to study such $\bb{Z}/p\bb{Z}$-covers of curves. Using Artin-Schreier root stacks, a classification of stacky curves with wild ramification of order $p$ can be achieved, which parallels the case of characteristic zero. 
	
	It is a natural question to study group actions on these stacks, which can be useful in the study of coverings from an equivariant perspective. It can also be potentially useful in the study of sheaves on these stacks. For example, over $\mathbb{C}$, vector bundles on root stacks can be understood through the data of parabolic vector bundles on the coarse moduli variety in many situations \cite{Bor07}. Group actions on root stacks over $\bb{C}$ were studied in \cite{CP25}, and as an application, an equivalence between equivariant connections on the root stack and equivariant parabolic connections on its coarse moduli variety was established. In this article, we aim to study a similar question in the positive characteristic setup. Given an action of a connected affine algebraic group on a smooth projective variety, we aim to provide a criterion for lifting the action on the Artin-Schreier root stacks of Kobin \cite{K21}. Let $X$ be a nonsingular projective variety over an algebraically closed field of characteristic $p>0$ with an action of a connected affine algebraic group $G$. Suppose $D$ is an effective Cartier divisor on $X$ invariant under the action of $G$; see Section \ref{sec:lift-of-action} for more details. The invertible sheaf $L:=\mc{O}(D)$  admits a tautological section $s\in H^0(X, L)$ whose divisor of zeros is  $D$. To introduce the data of ramification jump, fix a positive integer $m$ coprime to $p$, and another section $f\in H^0(X, L^{\otimes m})$ which has no common zeros with $s$. Consider the Artin-Schreier root stack $\wp_m^{-1}((L,s,f)/X)$ associated to this data; see Section \ref{section:preliminaries} for more details. Our first main result is the following. 
	\begin{theorem}[\text{Theorem \ref{thm:group-action-artin-schreier}}]
		Suppose $G$ has no nontrivial characters. Assume that the restriction of $f$ to the divisor $D$ is invariant for the induced $G$-linearization on $L^{\otimes m}$. Then the $G$-action on $X$ admits a lift to a $G$-action on the Artin-Schreier root stack $\mathcal{X}^{\nu}_m:= \wp_m^{-1}((L,s,f)/X)$. 
		%		Namely, there exists an action morphism
		%		$
		%		\sigma': G \times \mathcal{X}^{\nu}_m \longrightarrow \mathcal{X}^{\nu}_m$ such that the following diagram is $2$-commutative, where $\pi:\mc{X}^{\nu}_m\longrightarrow X$ denotes the coarse moduli map:
		%		\begin{align}
			%			\begin{gathered}
				%				\xymatrix{
					%					G \times \mathcal{X}^{\nu}_m \ar[r]^(.6){\sigma'} \ar[d]_{\text{Id}\times\pi} & \mathcal{X}^{\nu}_m \ar[d]^{\pi} \\
					%					G \times X \ar[r]^(.55){\sigma} & X
					%				}
				%			\end{gathered}
			%		\end{align}
	\end{theorem}
	
	\noindent
	In the final section, we discuss the existence of a tautological invertible sheaf $\mc{N}$ on the Artin-Schreier root stack $\mc{X}_m^{\nu}$ which serves as a $p$-th root of the pullback to $\mc{X}_m^{\nu}$ of the line bundle $L$ on $X$. We prove that, the $G$-action induced on $\mc{X}_m^{\nu}$ lifts to a $G$-linearization on this tautological invertible sheaf $\mc{N}$.
	
	\section{Preliminaries}\label{section:preliminaries}
	\subsection{Artin-Schreier root stacks}\hfill\\
	%	Ramification jumps play a crucial role in the study of $\bb{Z}/p\bb{Z}$-covers of curves in characteristic $p>0$. In order to incorporate the data of ramification jump into the Artin-Schreier root stacks, weighted projective lines are used. 
	Let $m\geq 1$ be a positive integer coprime to a given prime number $p$. The \textit{weighted projective line} $\bb{P}(1,m)$ is a stacky curve with coarse moduli space $\bb{P}^1$ whose only stacky point is $\infty := [0:1]$ with stabilizer group $\bb{Z}/m\bb{Z}$. Its $T$-valued points for a $K$-scheme $T$ is given by tuples $(\mc{L},s,f)$, where $\mc{L}\in \Pic(T)$, and $s$ and $f$ are sections of $\mc{L}$ and $\mc{L}^{\otimes m}$, respectively, such that they do not share any common zeros (see \cite[Corollary 6.7]{K21}). The coarse moduli map is given by  $\bb{P}(1,m)\longrightarrow \bb{P}^1$, $[x:y]\mapsto [x:y^m]$. 
	
	Consider the isogeny $\wp:\bb{G}_a \longrightarrow\bb{G}_a$\,, \,$x\mapsto x^p-x$. This extends to a cyclic $p$-cover $\Psi : \bb{P}(1,m)\longrightarrow\bb{P}(1,m)$\,, \,$[u:v]\mapsto[u^p:v^p-vu^{m(p-1)}]$. On the other hand, $\bb{G}_a$ acts on $\bb{P}(1,m)$ by the rule $x\cdot[u:v]:= [u:v+xu^m]$. It is easy to see that this action commutes with $\Psi$, and thus we get an induced morphism on the quotient stack $[\bb{P}(1,m)/\bb{G}_a]$. This is known as the \textit{universal Artin-Schreier cover with ramification jump m}, denoted by 
	$$\wp_m : [\bb{P}(1,m)/\bb{G}_a]\longrightarrow[\bb{P}(1,m)/\bb{G}_a]\,.$$
	
	According to \cite[Definition A.3]{AB21}, a locally Noetherian algebraic stack $\mc{X}$ is said to be \textit{normal} if there is a smooth surjection $U\longrightarrow \mc{X}$ where $U$ is a normal scheme. A representable morphism $\nu:\mc{X}^{\nu}  \longrightarrow \mc{X}$ from an algebraic stack $\mc{X}^{\nu}$ is said to be a \textit{normalization} of $\mc{X}$ if for any scheme $U$ and any smooth morphism $U \rightarrow \mc{X}$ , the
	pullback $\mc{X}^{\nu} \times_{\mc{X}} U \rightarrow U$ is the normalization of $U$. 	By \cite[Lemma A.4]{AB21}, the normalization of a Noetherian algebraic stack always exists, and it is unique upto unique isomorphism.
	
	\begin{definition}[\text{\cite[Definition 6.9]{K21}}]\label{def:artin-schreier-root-stack}
		Let $X$ be a $K$-scheme. Fix a positive integer $m$ coprime to $p$. Choose a line bundle $L$ on $X$, and two sections $s\in H^0(X,L)$ and $f\in H^0\left(X,L^{^{\otimes m}}\right)$ having no common zeros at any point of $X$. The  \textit{Artin-Schreier root stack of $X$ along }$(L,s,f)$, denoted by the symbol $\wp_{m}^{-1}((L,s,f)/X)$, is defined to be the normalized pullback of the diagram 
		\begin{align}\label{eqn:artin-schreier-diagram}
			\begin{gathered}
				\begin{tikzpicture}[xscale=4,yscale=2]
					\node at (0,1) (a) {$\wp_{m}^{-1}((L,s,f)/X)$};
					\node at (1,1) (b) {$[\bb{P}(1,m)/\bb{G}_{a}]$};
					\node at (0,0) (c) {$X$};
					\node at (1,0) (d) {$[\bb{P}(1,m)/\bb{G}_{a}]$};
					\draw[->] (a) -- (b);
					\draw[->] (a) -- (c);
					\draw[->] (b) -- (d) node[right,pos=.5] {$\wp_{m}$};
					\draw[->] (c) -- (d) node[above,pos=.5] {$\Phi_{(L,s,f)}$};
					\node at (.15,.7) {$\nu$};
					\draw (.1,.6) -- (.2,.6) -- (.2,.8);
				\end{tikzpicture}
			\end{gathered}
		\end{align}
		where $\wp_{m} : [\bb{P}(1,m)/\bb{G}_{a}]\rightarrow [\bb{P}(1,m)/\bb{G}_{a}]$ is the universal Artin--Schreier cover, and the bottom arrow $\Phi_{(L,s,f)}$ is the composition of the morphism $X\longrightarrow\bb{P}(1,m)$ arising from $(L,s,f)$ followed by the quotient map $\bb{P}(1,m)\rightarrow [\bb{P}(1,m)/\bb{G}_{a}]$.
	\end{definition}
	
	\begin{remark}\label{rem:valued-points}
		For a $K$-scheme $T$, it is difficult to give a nice description of the $T$-valued points of the Artin-Schreier root stack $\wp_{m}^{-1}((L,s,f)/X)$, partly due to the fact that it is a normalized pullback instead of an actual pullback. On the other hand, the actual (unnormalized) pullback of the diagram \eqref{eqn:artin-schreier-diagram}, namely $$X\times_{_{(L,s,f),\,[\bb{P}(1,m)/\bb{G}_{a}],\,\wp_m}}[\bb{P}(1,m)/\bb{G}_{a}]$$ 
		has a rather explicit description of its $T$-valued points,  which will be useful later. Let $D:= \dv(s)$ denote the divisor of zeros of $s$, which is a closed subscheme of $X$. It is shown in \cite[Remark 6.10]{K21} that the $T$-valued points are given by tuples of the form 
		$$\left(\varphi: T \longrightarrow X,\, M,\, t,\, g,\, \psi: M^{^{\otimes p}} \overset{\simeq}{\longrightarrow} \varphi^* L\right),$$
		where
		\begin{enumerate}[$\bullet$]
			\item $\varphi$ is a morphism,
			\item $M\in \Pic(T)$ and $t\in H^0(T,M)$,
			\item $g\in H^0\left(mE,M^{^{\otimes m}}|_{_{mE}}\right)$ where $E:= \dv(t)$, and
			\item $\psi: M^{^{\otimes p}}\overset{\simeq}{\longrightarrow} \varphi^*L$ is an $\mc{O}_T$-linear isomorphism such that $\psi(t^{^{\otimes p}}) = \varphi^*s$. Moreover, looking at the resulting Cartesian diagram \begin{align}\label{diagram-1}
				\begin{gathered}
					\xymatrix{mpE \ar[rr]^{\varphi_{_{mD}}} \ar@{^{(}->}[d] && mD \ar@{^{(}->}[d] \\ 
						T \ar[rr]^{\varphi} && X
					}
				\end{gathered}
			\end{align}
			%			arising from the equalities of divisors
			%			$$pE = \text{div}(t^{^{\otimes p}}) = \text{div}(\varphi^* s) = \varphi^*(\dv(s)) = \varphi^* D\,,$$
			the restriction $\psi_{_{mpE}}: M^{^{\otimes mp}}|_{_{mpE}} \overset{\simeq}{\longrightarrow} \varphi_{_{mD}}^*(L^{{^{\otimes m}}}|_{_{mD}})$ is required to satisfy				\begin{align}\label{eqn:eq-1}
				\psi_{_{mpE}}^{^{\otimes m}}\left(g^p - t^{m(p-1)}g\right) = \varphi_{_{mD}}^* (f|_{_{mD}}).
			\end{align} 
		\end{enumerate}
	\end{remark}
	% In Proposition \ref{prop:T-points}, we shall see another equivalent description of the $T$-valued points in a more general setup of Artin-Schreier-Witt case. The equivalence of these two descriptions is discussed in Remark \ref{rem:T-points-comparison}.
	
	\begin{remark}\label{rem:special-atlas}
		Let us mention a special case when the valued points of the Artin-Schreier root stack can be described explicitly. Let $u:U\longrightarrow \mc{X}_m$ be a smooth atlas for the unnormalized stack, where $U$ is a scheme. Consider the fiber product $U^{\nu} := U\times_{\mc{X}_m}\mc{X}_m^{\nu}$\,. Since $\mc{X}^{\nu}_m\rightarrow\mc{X}_m$ is representable,  $U^{\nu}$ is again a scheme; moreover, it is the  normalization of $U$. Also, the pullback morphism $u^{\nu}: U^{\nu}\longrightarrow \mc{X}_m^{\nu}$ is again a smooth atlas for $\mc{X}_m^{\nu}$. Clearly, the map $u^{\nu}$ is essentially uniquely determined by the composition $U^{\nu}\xrightarrow{\,\,h\,\,} U \xrightarrow{\,\,u\,\,} \mc{X}_m$ due to properties of normalizations. Now, as in Remark \ref{rem:valued-points}, if $u$ is given by the tuple 
		$$\bigl(\varphi': U \longrightarrow X,\, M',\, t',\, g',\, \psi': M'^{^{\otimes p}} \overset{\simeq}{\longrightarrow} \varphi'^* L\bigr),$$ 
		then the composition $U^{\nu}\xrightarrow{\,\,h\,\,} U \xrightarrow{\,\,u\,\,} \mc{X}_m$ is given by the tuple
		$$\bigl(h^*\varphi': U^{\nu} \longrightarrow X,\, h^*M',\, h^*t',\, h^*g',\, h^*\psi': M'^{^{\otimes p}} \overset{\simeq}{\longrightarrow} (h\circ\varphi')^*L \bigr)\,.$$
		This tuple uniquely determines the map $u^{\nu}: U^{\nu}\longrightarrow\mc{X}^{\nu}_m$ above.
		
		Thus, for any smooth atlas $V$ for $\mc{X}^{\nu}_m$ which can be obtained as the pullback of a smooth atlas for $\mc{X}_m$, the $V$-valued points of $\mc{X}^{\nu}_m$ can be described explicitly.
	\end{remark}
	\section{Lifting group action to Artin-Schreier root stack}\label{sec:lift-of-action}
	In what follows, $K$ denotes an algebraically closed field of characteristic $p>0$. The term \textit{variety} over $K$ means an integral separated scheme of finite type over $\Spec K$. 
	Let $G$ be a connected affine algebraic group over  $K$. Suppose $X$ is a nonsingular projective 
	variety with a $G$-action $\sigma: G\times X \to X$. Let $p_2: G\times X \to X$ be the projection map 
	onto the second factor. Recall that a \textit{linearization} of the $G$-action $\sigma$ on $X$ 
	to an invertible sheaf $L$ on $X$ is an isomorphism 
	$$\Phi : p_2^*L \stackrel{\simeq}{\longrightarrow} \sigma^*L$$ 
	of sheaves of $\mc O_{G\times X}$-modules on $G\times X$ satisfying the following 
	cocycle condition: 
	\begin{equation*}
		(\mu\times\Id_X)^*\Phi = (\Id_G\times\sigma)^*\Phi\circ p_{23}^*\Phi,
	\end{equation*}
	where $p_{23} : G\times G\times X \to G\times X$ is the projection morphism onto the 
	second and third factors, and $\mu : G\times G \to G$ is the group operation on $G$. For $G$ connected, the above cocycle condition 
	is automatic for an invertible sheaf $L$ on $X$; see \cite[\S\,7.2, Lemma 7.1]{Dol}.
	
	Let us assume that $G$ has no non-trivial characters. This holds, for example, for any group which satisfies  $G=[G, G]$; examples include connected semisimple groups. Suppose $X$ admits an effective Cartier divisor $D$ which is invariant under the $G$-action, in the sense that $p_2^* D = \sigma^* D$. The line bundle $\mc{O}(D)$ admits a canonical global section $s\in H^0(X,\mc{O}(D))$ satisfying $\dv(s) =D$.
	%	Let $L \in \text{Pic}(X)$, and $s \in H^0(X, L)$. 
	%	Let $D = \text{div}(s)$ be the associated Cartier divisor (see \cite[p. 157]{HartshorneAG}). The ideal sheaf $\mc{O}(-D)\subset \mc{O}_X$ gives rise to a closed subscheme of $X$, which we shall again denote by the same symbol $D$. 
	%	
	\begin{lemma}\label{lem:linearization}
		Suppose $G$ has no non-trivial characters. Given a $G$-invariant effective Cartier divisor $D$ as above, one can construct a $G$-linearization
		\begin{align}
			\phi: p_2^* \mc{O}(D) \overset{\simeq}{\longrightarrow} \sigma^* \mc{O}(D)\,\,\,\text{satisfying}\,\,\,\phi(p_2^*s) = \sigma^*(s)\,.
		\end{align}
		In other words, $s$ is an invariant section of the $G$-linearized line bundle $L:=\mc{O}(D)$.
	\end{lemma}
	\begin{proof}
		This follows from the same argument as in \cite[Lemma 2.2.2]{CP25}; the same proof works over any algebraically closed field. 
	\end{proof}
	\begin{example}
		Let $G\ =\  \bb{G}_m$ be the multiplicative group over $K$. For an integer $d$, the action of $\bb{G}_m$ on $\bb{P}^1_K$ given by 
		\begin{align}
			\sigma:\bb{G}_m\times\bb{P}^1_K&\longrightarrow\bb{P}^1_K\\
			(t,[x:y])&\mapsto [t^dx:y]
		\end{align}
		fixes the point $[0:1]$. Let $D$ be the reduced effective divisor given by the point $[0:1]$. One can show that $\Phi_t(s_D) = t^{-d}\sigma_t^*(s_D)$, where $t\mapsto t^{-d}$ is a non-trivial character of $\bb{G}_m$. Thus, the condition of $G$ having no nontrivial characters is necessary in Lemma \ref{lem:linearization}.
	\end{example}
	%	Now, since $$\text{div}(s) =\text{div}(s_D) =D,$$
	%	one can construct an $\mc{O}_X$-linear isomorphism $\eta:\mc{O}(D)\simeq L$ satisfying $\eta(s_D) =s$. This $\eta$ enables one to transfer the $G$-linearization from $\mc{O}(D)$ to $L$, thus ensuring that there exists a $G$-linearization
	%	\begin{align}\label{eqn:eq-3}
		%		\phi:p_2^*L\overset{\simeq}{\longrightarrow} \sigma^*L \ \ \text{satisfying} \ \ \phi(p_2^* s) = \sigma^* s\ .
		%	\end{align} 
	
	The following result --which only appears in the arxiv version of the paper \cite{AB21}, not in the published version -- shall be useful for our purpose. Recall that a stack $\mc{X}$ has an associated topological space of points denoted by $|\mc{X}|$, whose underlying set consists of equivalence classes of morphisms $x: \text{Spec}\,k\rightarrow \mc{X}$ for a field $k$,  where two morphisms $x: \Spec\,k\rightarrow \mc{X}$ and $x':\Spec\,k'\rightarrow\mc{X}$ are considered equivalent if there exists a field extension $L\supset k,k'$ such that $x|_{\Spec L}$ and $x'|_{\Spec L}$ are 2-isomorphic. 
	%	and a 2-commutative diagram
	%	\begin{align*}
		%	{\Small	\begin{gathered}
				%			\xymatrix{
					%				& \Spec\,k \ar[dr]_{x} & \\
					%				\Spec\,L \ar[ur] \ar[dr] & & \mc{X}\\
					%				& \Spec\,k' \ar[ur]^{x'} &
					%			}
				%		\end{gathered}}
		%	\end{align*}
	\begin{proposition}[\text{\cite[Proposition A.6]{AB21}}]\label{prop:normalization-morphism}
		Let $\mathcal{X}$ and $\mathcal{Y}$ be locally Noetherian algebraic stacks. Suppose that for each normal scheme $T$, functors $
		F_T: \mathcal{X}(T) \longrightarrow \mathcal{Y}(T)$
		are compatible with base change, and the induced morphism on points $|F|: |\mathcal{X}| \to |\mathcal{Y}|$ is dominant on irreducible components. Then $F_T$ induces a unique representable morphism $F^\nu: \mathcal{X}^\nu \to \mathcal{Y}^\nu$ on normalizations.
	\end{proposition}
	Fix a positive integer $m$ coprime to $p$, and let $f \in H^0\bigl(mD, L^{^{\otimes m}}|_{_{mD}}\bigr)$.
	Due to the $G$-invariance of $D$, the $G$-action $\sigma$ on $X$ restricts to a $G$-action on $mD$, which we denote by $\sigma_{mD}$:
	\begin{align}\label{eqn:eq-2}
		\sigma_{_{mD}}: G \times mD \longrightarrow mD.
	\end{align}
	Similarly, for $L=\mc{O}(D)$, the $G$-linearization $\phi: p_2^* L \overset{\simeq}{\longrightarrow} \sigma^* L$ restricts to a $G$-linearization 
	\begin{align}\label{eqn:linearization-on-restriction}
		\phi_{_{mD}}: p_2^* \left(L|_{_{mD}}\right) \overset{\simeq}{\longrightarrow} \sigma_{_{mD}}^* \left(L|_{_{mD}}\right).
	\end{align}
	Given a triple $(L, s, f)$ as in Definition \ref{def:artin-schreier-root-stack}, consider the actual (unnormalized) pullback of the diagram \eqref{eqn:artin-schreier-diagram}: 
	\begin{align}\label{eqn:unnormalized-root-stack}
		\mc{X}_m :=X\times_{_{(L,s,f),[\bb{P}(1,m)/\bb{G}_{a}],\wp_m}}[\bb{P}(1,m)/\bb{G}_{a}]
	\end{align}
	Since the Artin-Schreier root stack $\wp_m^{-1}((L,s,f)/X)$ is the normalization of $\mc{X}_m$ (see Definition \ref{def:artin-schreier-root-stack}), it is denoted by 
	$\mathcal{X}^{\nu}_m := \wp_m^{-1}((L,s,f)/X)$ below. We can now state the main result of this section. By a slight abuse of notation, we shall not distinguish between a $K$-scheme and its associated stack.
	\begin{theorem}\label{thm:group-action-artin-schreier}
		Let $G$ be a connected affine algebraic group with no non-trivial characters. Suppose $G$ acts on a smooth projective variety $X$ via $\sigma: G\times X\longrightarrow X$. Suppose $X$ admits a $G$-invariant effective Cartier divisor $D$, and let $s$ be the canonical section of $L:=\mc{O}(D)$ satisfying $\dv(s) =D$. Recall the  $G$-linearization $\phi$ on $L$ from Lemma \ref{lem:linearization}, and let $f\in H^0(X,L^{^{\otimes m}})$ be a section having no common zeros with $s$, such that $f|_{_{mD}}$ is an invariant section of $L^{^{\otimes m}}|_{_{mD}}$ for the induced linearization $\phi_{_{mD}}^{^{\otimes m}}$\,; namely, with $\phi_{_{mD}}$ and $\sigma_{_{mD}}$  defined as in \eqref{eqn:linearization-on-restriction} and \eqref{eqn:eq-2} respectively, we have \begin{align}\label{eqn:invariant-section}
			\phi_{_{mD}}^{^{\otimes m}}(p_2^* f|_{_{mD}}) = \sigma_{_{mD}}^* (f|_{_{mD}})\,.
		\end{align}  
		Then the $G$-action on $X$ admits a lift to a $G$-action on the Artin-Schreier root stack $\mathcal{X}^{\nu}_m$. 
		Namely, there exists an action morphism
		$
		\sigma': G \times \mathcal{X}^{\nu}_m \longrightarrow \mathcal{X}^{\nu}_m$ such that the following diagram is $2$-commutative, where $\pi:\mc{X}^{\nu}_m\longrightarrow X$ denotes the coarse moduli map:
		\begin{align}\label{diagram-2}
			\begin{gathered}
				\xymatrix{
					G \times \mathcal{X}^{\nu}_m \ar[r]^(.6){\sigma'} \ar[d]_{\text{Id}\times\pi} & \mathcal{X}^{\nu}_m \ar[d]^{\pi} \\
					G \times X \ar[r]^(.55){\sigma} & X
				}
			\end{gathered}
		\end{align}
	\end{theorem}

	\begin{proof}
		Our method of producing the required $G$-action on $\mathcal{X}^{\nu}_m$ is to first construct a $G$-action on the unnormalized pullback
		$$\alpha: G\times \mc{X}_m \longrightarrow \mc{X}_m$$  and then make use of Proposition \ref{prop:normalization-morphism}. The steps are described below.
		
		\noindent
		\underline{\textbf{Action of $\alpha$ on objects:}}\ \   Let $T$ be a $K$-scheme.
		Let $h \in G(T)$, so that $h: T \rightarrow G$ is a morphism. Also, consider an object $\tau$ in $\mathcal{X}_m(T)$ given by the tuple (see Remark \ref{rem:valued-points})
		\[\tau := \bigl(\varphi: T \longrightarrow X,\, M,\, t,\, g,\, \psi: M^p \overset{\simeq}{\longrightarrow} \varphi^* L\bigr)\,.\]
		Here $M \in \text{Pic}(T)$, $t \in H^0(T, M)$, and $g \in H^0\bigl(mE, M^{^{\otimes m}}|_{_{mE}}\bigr)$, where $E := \text{div}(t)$. Furthermore, we have $\psi(t^{^{\otimes p}}) = \varphi^*s$, and the equation \eqref{eqn:eq-1} is also satisfied. 
		
		Starting from an object $(h,\tau) \in (G\times \mc{X}_m)(T) = G(T)\times \mathcal{X}_m(T)$, we would like to produce another object in $\mathcal{X}_m(T)$, which we shall denote by $\alpha_T(h,\tau)$.
		
		Let $u := (h, \varphi): T \longrightarrow G \times X$.
		We had earlier observed that $L$ admits a  $G$-linearization $\phi: p_2^* L \overset{\simeq}{\longrightarrow} \sigma^* L$ satisfying $\phi(p_2^*s)=\sigma^*s$. The composition 
		$$ M^{^{\otimes p}}\underset{\simeq}{\xrightarrow{\,\,\quad\psi\,\,\quad}} \varphi^* L = u^* p_2^* L \underset{\simeq}{\xrightarrow{\,\,\quad u^*\phi\,\,\quad}} u^* \sigma^* L = (\sigma \circ u)^* L$$
		leads to an $\mc{O}_X$-linear isomorphism $(u^*\phi) \circ \psi: M^{^{\otimes p}} \xrightarrow{\,\,\,\simeq\,\,\,} (\sigma\circ u)^*L\,.$
		
		Next, we claim that the following tuple is an object of the groupoid $\mathcal{X}_m(T)$:
		\begin{align}\label{eqn:eq-5}
			\alpha_T(h,\tau):=\bigl(\sigma \circ u: T \to X, M,\, t,\, g,\, (u^*\phi)\circ  \psi: M^p \overset{\simeq}{\longrightarrow} (\sigma \circ u)^* L\bigr)
		\end{align} 
		For this to hold, we must check that all the conditions of Remark \ref{rem:valued-points} hold. First one:
		{\small 
			\begin{align*}
				\left((u^* \phi) \circ \psi\right)(t^{^{\otimes p}}) = (u^* \phi)(\psi(t^{^{\otimes p}})) 
				= (u^* \phi)(\varphi^* s)
				&= (u^* \phi)(u^* p_2^* s)\qquad{\text{because}\,\,\varphi = p_2\circ u}\\[4pt]
				&= u^*\left(\phi(p_2^* s)\right) & \\[4pt]
				&= (\sigma \circ u)^* s \qquad {\text{by invariance assumption}.}
			\end{align*}
		}
		Next, we check the validity of equation \eqref{eqn:eq-1}. Consider the following Cartesian diagram, which easily follows from diagram \eqref{diagram-1}:
		\begin{align}\label{diagram-3}
			{\Small	\begin{gathered}
					\xymatrix{
						mpE \ar@/^2pc/[rrr]^{\varphi_{_{mD}}}\ar[rr]^{u_{_{mD}}} \ar@{^{(}->}[d]&& G\times mD \ar[r]^(.55){p_2} \ar@{^{(}->}[d] & mD \ar@{^{(}->}[d] \\
						T \ar@/_2pc/[rrr]^{\varphi} \ar[rr]^{u} && G\times X \ar[r]^(.55){p_2} & X
					}
			\end{gathered}}
		\end{align}
		\noindent
		
		It follows that
		{\small \begin{align*}
				&\left((u^*\phi)\circ \psi\right)_{mpE}^{\otimes m}
				\bigl(g^p - t^{m(p-1)}g\bigr)
				\\[4pt]
				=&\left(u^*\phi\right)_{mpE}^{\otimes m}\Bigl(\psi_{mpE}^{\otimes m}
				\bigl(g^p - t^{m(p-1)}g\bigr)
				\Bigr)\\[4pt]
				=&\left(u^*\phi\right)_{mpE}^{\otimes m}\bigl(\varphi_{_{mD}}^* f|_{_{mD}}\bigr)
				\quad\, \text{by \eqref{eqn:eq-1}} \\[4pt]
				=&
				\left(u^*\phi\right)_{mpE}^{\otimes m}
				\bigl(u_{_{mD}}^* p_2^* f|_{_{mD}}\bigr)
				\quad \,\text{by diagram \eqref{diagram-3}} \\[4pt]
				=&
				\Bigl(u_{_{mD}}^*(\phi_{_{mD}}^{\otimes m})\Bigr)
				\bigl(u_{_{mD}}^* p_2^* f|_{_{mD}}\bigr) \\[4pt]
				=&
				u_{_{mD}}^*\left(\phi_{_{mD}}^{\otimes m}(p_2^* f|_{_{mD}})\right) \\[4pt]
				=&
				u_{_{mD}}^*(\sigma_{_{mD}}^* f|_{_{mD}})
				\qquad \text{by invariance assumption \eqref{eqn:invariant-section}} \\[4pt]
				=&
				(\sigma \circ u)_{_{mD}}^* (f|_{_{mD}}) .
			\end{align*}
		}
		% {
			% 	\begin{align*}
				% 		\left((u^* \phi) \circ \psi\right)_{_{mpE}}^{^{\otimes m}}\left(g^p&-t^{m(p-1)}g\right)
				% 		=  \left(u^* \phi\right)_{_{mpE}}^{^{\otimes m}}\left(\psi^{^{\otimes m}}_{_{mpE}}(g^p-t^{m(p-1)}g)\right)
				% 		\underset{\eqref{eqn:eq-1}}{=} \left(u^* \phi\right)_{_{mpE}}^{^{\otimes m}}(\varphi_{_{mD}}^* f)\\
				% 		&\underset{\eqref{diagram-3}}{=} \left(u^* \phi\right)_{_{mpE}}^{^{\otimes m}}(u_{_{mD}}^*p_2^*f) \underset{\eqref{diagram-3}}{=} \left(u_{_{mD}}^*\left(\phi_{_{mD}}^{^{\otimes m}}\right)\right)\left(u_{_{mD}}^*p_2^*f\right)\\
				% 		%			&= \left(u^* \phi\right)_{_{mpE}}^{^{\otimes m}}(\varphi_{_{mD}}^* f) \\[4pt]
				% 		&= u_{_{mD}}^* (\phi_{_{mD}}^{^{\otimes m}} (p_2^* f)) = u_{_{mD}}^* \left(\sigma_{_{mD}}^* f\right) \qquad\quad[\text{by assumption}]\\
				% 		&\qquad\qquad\qquad\qquad= (\sigma \circ u)_{_{mD}}^* f.
				% 	\end{align*}
			% }
		Hence $\alpha_T(h,\tau)$ in \eqref{eqn:eq-5} is a valid object in $\mc{X}_m(T)$.
		
		\noindent
		\underline{\textbf{Action of $\alpha$ on arrows:}}\ \  Consider two objects in $\mc{X}_m$ (see Remark \ref{rem:valued-points}):
		\begin{align}
			\tau:=(\varphi:T\to X,\,M,\,t,\,g,\,\psi)\ \ \text{and}\ \ \tau':=(\varphi':T'\to X,\,M',\,t',\,g',\,\psi')\,.
		\end{align} 
		Let $h\in G(T)$ and $h'\in G(T')$. An arrow between two objects $(h,\tau)$ and $(h',\tau')$ in $G\times \mc{X}_m$ is given by a pair $\left(\xi,\,(\ell,\theta)\right)$, where 
		\begin{enumerate}[$\bullet$]
			\item $\xi:T\longrightarrow T'$ is a morphism of $K$-schemes satisfying $h'\circ\xi=h$, and 
			\item $(\ell,\theta) : \tau \longrightarrow \tau'$ is a morphism in $\mc{X}_m$, i.e. $\ell: T\longrightarrow T'$ is a morphism of $K$-schemes such that $\varphi'\circ \ell = \varphi$, and $\theta: M \overset{\simeq}{\longrightarrow} \ell^*(M')$ is an $\mc{O}_T$-linear isomorphism satisfying $\theta(t)=\ell^*(t')$ and $\theta^{\otimes m}(g)=\ell^*(g')$.
		\end{enumerate}
		By the description of fiber product of stacks, $\xi$ and $(\ell,\theta)$ are both required to lie over the \textit{same} arrow in $\text{Sch}/K$, which implies $\xi = \ell$. This gives $\varphi'\circ \xi = \varphi'\circ\ell=\varphi$. Next, consider the maps $$u:= (h,\varphi): T\longrightarrow G\times X\ \ \text{and}\ \  u':=(h',\varphi'): T'\longrightarrow G\times X.$$
		From $h'\circ \xi = h$ and $\varphi'\circ \xi =\varphi$, it follows that $u'\circ\xi =u$. 
		It is now straightforward to show that the pair $(\ell,\theta)$ gives rise to an arrow
		{\small
			\begin{align*}
				\alpha_T(h,\tau) = \left(\sigma\circ u,\,M,\,t,\,g,\,(u^*\phi)\circ\psi\right) \xrightarrow{\,\,\,\,(\ell,\theta)\,\,\,\,} \alpha_{T'}(h',\tau') = \left(\sigma\circ u',\,M',\,t',\,g',\,(u'^*\phi)\circ\psi'\right).
			\end{align*}
		}
		\noindent
		In particular, taking $T=T'$, we get a functor between the respective fiber groupoids:
		\[
		\alpha_T: (G \times \mathcal{X}_m)(T) = G(T) \times \mathcal{X}_m(T) \longrightarrow \mathcal{X}_m(T).
		\]
		We have thus produced a morphism of stacks $\alpha: G\times \mc{X}_m\longrightarrow \mc{X}_m$, such that the following diagram is $2$-commutative (see \cite[Proposition 2.2.4]{CP25} for more details):
		\begin{align*}
			\begin{gathered}
				\xymatrix{
					G \times \mathcal{X}_m \ar[r]^(.6){\alpha} \ar[d] & \mathcal{X}_m \ar[d] \\
					G \times X \ar[r]^(.55){\sigma} & X
				}
			\end{gathered}
		\end{align*}
		In particular, we have obtained the functors $\alpha_T$ on the fiber groupoids over \textit{normal} schemes $T$. In order to apply Proposition \ref{prop:normalization-morphism}, it only remains to show that the induced map on points $|\alpha|:|G \times \mathcal{X}_m| \to |\mathcal{X}_m|$ is dominant on irreducible components. The equality $$|G \times \mathcal{X}_m| = |G| \times |\mathcal{X}_m|$$ combined with the irreducibility of $G$, implies that the irreducible components of the space $|G\times \mc{X}_m|$ are of the form $|G|\times V,$ where $V$ is an irreducible component of $|\mc{X}_m|$. Now, for the map $|\alpha|$ to be dominant on irreducible components, it is enough to show that for each irreducible component $V$ of $|\mc{X}_m|$, the component $|G|\times V$ is sent onto $V$ via $|\alpha|$. To see this,  note that the subspace $\{e\}\times V$ maps onto $V$ under $|\alpha|$, where $e$ is the identity element in $|G|$. As the image of an irreducible subset under a continuous map is again irreducible, it follows that the image of $|G|\times V$ under the map $|\alpha|$ is again an irreducible subset of $|\mc{X}_m|$ which moreover contains the irreducible component $V$. This forces the image of $|G|\times V$ under $|\alpha|$ to be equal to $V$. 
		%	We also know that $\mathcal{X}_m \to X$ induces a bijection on irreducible components \cite[Lemma 4.5]{AB21}.
		%	Consider the coarse moduli map $\pi: \mathcal{X}_m \to X$. Recall that $X$ is irreducible.
		%	It is a fact that $\pi$ is integral and surjective. Thus $\mathcal{X}_m$ is irreducible (or rather, the bijection implies it if $X$ is).
		This proves our claim, namely that  $|\alpha|:|G \times \mathcal{X}_m| \longrightarrow |\mathcal{X}_m|$ is dominant on irreducible components. Moreover, since $G$ is smooth and hence normal, we have
		$(G \times \mathcal{X}_m)^\nu \cong G \times \mathcal{X}_m^\nu\,.$
		Thus, Proposition \ref{prop:normalization-morphism} applied on $\alpha:G\times\mc{X}_m\longrightarrow \mc{X}_m$ produces a morphism $$\sigma':G \times \mathcal{X}_m^\nu \longrightarrow \mathcal{X}_m^\nu\,.$$
		The $2$-commutativity of the diagram \eqref{diagram-2} is a direct consequence of the universal property of normalization. It is a routine verification to check that $\alpha$ gives rise to a group action, from which it is easy to see that $\sigma'$ gives rise to a group action as well.
	\end{proof}
	\section{Tautological invertible sheaf  and the lifting of $G$-action}
	As before, let $X$ be a nonsingular projective 
	variety with a $G$-action $\sigma : G\times X \to X$, and moreover assume that $G$ has no non-trivial characters. In this section, we aim to describe a certain tautological invertible sheaf on the Artin-Schreier root stack $\mc{X}^{\nu}_m$, and show that the induced $G$-action on $\mc{X}_m^{\nu}$ coming from Theorem \ref{thm:group-action-artin-schreier} lifts to a $G$-linearization on the invertible sheaf. We shall begin our discussion with the case of the unnormalized pullback $\mc{X}_m$ in \eqref{eqn:unnormalized-root-stack} associated to a triple $(L,s,f)$. As in Remark \ref{rem:valued-points}, let $u: U\rightarrow \mc{X}_m$ be an atlas given by the tuple $(\varphi_u: U\rightarrow X, \,M_u,\,t_u,\,g_u\,,\psi_u)\in\mc{X}_m(U)$. Let $v: V\longrightarrow \mc{X}_m$ be another atlas given by the tuple $(\varphi_v: V\rightarrow X, \,M_v,\,t_v,\,g_v\,,\psi_v)$, which fit in a 2-commutative diagram of stacks:
	\begin{align}
		\xymatrix{
			U \ar[rd]^{u} \ar[d]_{\gamma}& \\
			V \ar[r]^{v} & \mc{X}_m
		}
	\end{align}
	By definition, such a data gives rise to an $\mc{O}_U$-linear isomorphism of invertible sheaves
	\[\psi_{u,v}: \gamma^*M_v\xrightarrow{\,\,\simeq\,\,} M_u\,.\]
	It follows that the invertible sheaves $M_u$ on each such atlas $U$ can be patched together to give rise to an invertible sheaf $\mc{M}$ on $\mc{X}_m$. If $p: \mc{X}_m\longrightarrow X$ denotes the projection, the construction produces a natural isomorphism of invertible sheaves $$\mc{M}^{\otimes p}\xrightarrow{\,\,\simeq\,\,}
	p^*L\,.$$ 
	We shall first construct a $G$-linearization on $\mc{M}$, which lifts the induced $G$-action on $\mc{X}_m$ coming from a $G$-action on $X$. The following result will be useful for our purpose.
	\begin{proposition}
		Let $u : U \to \mc{X}_m$ be an atlas of $\mc X$ 
		given by the tuple $(\varphi : U \to X,\, M,\, t,\, g,\, \psi)$. 
		Then the Cartesian diagram of schemes 
		\begin{equation*}\label{diag:cartesian-diagram-1}
			\xymatrix{
				U' \ar[rr]^-{\sigma'} \ar[d]_{h} && U\ar[d]^{\varphi} \\
				G\times X \ar[rr]^-{\sigma} && X 
			}
		\end{equation*}
		induces a Cartesian diagram of the form
		\begin{equation}\label{diag:Cartesian-diagram-for-stack}
			\begin{gathered}
				\xymatrix{
					U' \ar[rr]^{\sigma'} \ar[d]_{u'} && U \ar[d]^{u} \\
					G\times\mc{X}_m \ar[rr]^{\alpha} && {\mc{X}_m} 
				}
			\end{gathered}
		\end{equation}
		where $\alpha:G\times \mc{X}_m\rightarrow \mc{X}_m$ is the $G$-action constructed in Theorem \ref{thm:group-action-artin-schreier}.
	\end{proposition}
	\begin{proof}
		The proof is along the same lines as that of \cite[Proposition 2.3.1]{CP25}. Suppose $h : U' \to G\times X$ 
		is given by the pair of morphisms $(a : U' \to G, \beta : U' \to X)$\,. We claim that the tuple 
		\begin{align}\label{eqn:tuple}
			u_0 :=\left(\beta : U' \to X,\, (\sigma')^*(M),\, (\sigma')^*(t),\, (\sigma')^*(g),\,  
			h^*(\phi^{-1})\circ(\sigma')^*\psi \right)
		\end{align}
		is a well-defined object of $\mc{X}_m$. Namely, we need to check whether the conditions of Remark \ref{rem:valued-points} are satisfied. All of the conditions except Equation \eqref{eqn:eq-1} follow from the proof of \cite[Proposition 2.3.1]{CP25}. It only remains to check whether Equation \eqref{eqn:eq-1} is satisfied in this case. We shall adopt the following notations for the elements in the tuple \eqref{eqn:tuple} for the sake of convenience:
		$$M' :=(\sigma')^*(M),\,t':= (\sigma')^*(t),\,g':=(\sigma')^*(g),\,\psi':=h^*(\phi^{-1})\circ(\sigma')^*\psi\,.$$
		As before, we denote $D=\dv(s)$ and $E=\dv(t)$. Let $E' := \dv(t')$, and consider the following commutative diagram:
		\begin{align*}
			{\Small \begin{gathered}
					\xymatrix{ mpE' \ar[rr]^{\beta_{_{mD}}} \ar@{_(->}[d] && mD \ar@{_(->}[d] \\
						U'\ar[rr]^{\beta} && X
					}
			\end{gathered}}
		\end{align*}
		We need to check whether the equation $(\psi')^{^{\otimes m}}|_{_{mpE'}} \left((g')^p - (t')^{m(p-1)}(g')\right) = \beta_{_{mD}}^* (f |_{_{mD}})$ holds.
		We shall use the following diagram, in which every square is commutative.
		\[{\Small
			\xymatrix@R=2.5pc@C=2pc{
				mpE' \ar[rr]^{ \widehat{\sigma}} \ar@{_(->}[d] \ar[ddr]^(.8){h_{_{mD}}} & & mpE \ar@{_(->}[d] \ar[ddr]^{\varphi_{mD}} \\
				U' \ar[rr]^(.6){\sigma'} \ar[dd]_h & & U \ar[dd]_(.3)\varphi \\
				& G \times mD \ar[rr]_(.4){\sigma_{mD}} \ar@{_(->}[dl] & & mD \ar@{_(->}[dl] \\
				G \times X \ar[rr]^\sigma & & X
		}}
		\]
		Let us compute:
		{\small \begin{align*}
				(\psi')^{^{\otimes m}}|_{_{mpE'}} \left(g'^p - (t')^{m(p-1)}(g')\right) &= h^*(\phi^{-1})^{^{\otimes m}}|_{_{mpE'}} \left( (\sigma'^* \psi)^{^{\otimes m}}|_{_{mpE'}} (g'^{p} - (t')^{m(p-1)}g') \right) \\
				&= h^*(\phi^{-1})^{^{\otimes m}}|_{_{mpE'}} \left( \widehat{\sigma}^* (\psi^{^{\otimes m}}|_{_{mpE}}) (\widehat{\sigma}^* (g^p - t^{m(p-1)}g)) \right) \\
				&= h^*(\phi^{-1})^{^{\otimes m}}|_{_{mpE'}} \left( \widehat{\sigma}^* (\psi^{^{\otimes m}}|_{_{mpE}} (g^p - t^{m(p-1)}g)) \right) \\
				&= h^*(\phi^{-1})^{^{\otimes m}}|_{_{mpE'}} (\widehat{\sigma}^* \varphi_{_{mD}}^* (f|_{_{mD}})) \\
				&= h^*(\phi^{-1})^{^{\otimes m}}|_{_{mpE'}}(h^*_{_{mD}}\sigma^*_{_{mD}}(f|_{_{mD}}))\\
				&= h^*_{_{mD}}(\phi^{-1}_{_{mD}})^{^{\otimes m}} (h^*_{_{mD}}\sigma^*_{_{mD}}(f|_{_{mD}}))\,,\,\,\,\phi_{_{mD}}\,\,\text{as in } \eqref{eqn:linearization-on-restriction}\\
				&= h^*_{_{mD}}\left((\phi^{-1}_{_{mD}})^{^{\otimes m}}\sigma^*_{_{mD}}(f|_{_{mD}})\right)\\
				%			&= h^*((\phi^{-1})^{^{\otimes m}}(\sigma^* f))|_{_{mpE'}} \\
				%			&= h^*_{_{mD}}\left((\phi^{-1})^{^{\otimes m}}(\sigma^* f)|_{_{G \times mD}}\right) \\
				%			&= h^*_{_{mD}}\left((\phi_{_{mD}}^{-1})^{^{\otimes m}}(\sigma_{_{mD}}^*(f|_{_{mD}}))\right) \\
				&= h^*_{_{mD}}\left(p_2^*(f|_{_{mD}})\right)\,,\,\quad{\small \text{by the invariance assumption on }}\,f|_{_{mD}} \\
				%&= ((p_2^* f)|_{_{G \times mD}})|_{_{mpE'}}\\
				& = (h^* p_2^* f)|_{_{mpE'}}\\
				& = (\beta^* f )|_{_{mpE'}} = \beta_{_{mD}}^* (f |_{_{mD}})\,.
		\end{align*}}
		Thus, $u_0$ in \eqref{eqn:tuple} is a valid element of $\mc{X}_m(U')$. Recall the map $a:U'\rightarrow G$, and define 
		$$u':= (a,u_0): U'\longrightarrow G\times \mc{X}_m\,.$$
		It now follows from an exact similar argument as in \cite[Proposition 2.3.1]{CP25} that $u'$ makes the diagram \eqref{diag:Cartesian-diagram-for-stack} Cartesian.
	\end{proof}
	
	\begin{proposition}
		The $G$-action on $\mc{X}_m$ as constructed in the proof of Theorem \ref{thm:group-action-artin-schreier} can be lifted to a $G$-linearization on $\mc{M}$.
	\end{proposition}
	\begin{proof}
		This follows from an exact similar argument as in \cite[Proposition 2.3.11]{CP25}.
	\end{proof}
	
	\noindent
	We are in a position to state and prove the final result of this section. Recall that $X$ is a nonsingular projective variety with a $G$-action $\sigma : G\times X \to X$, and moreover assume that $G$ has no non-trivial characters. Also, $\pi:\mc{X}_m^{\nu} \to X$ denotes the coarse moduli map in what follows.
	
	\begin{theorem}
		There exists a tautological invertible sheaf $\mc{N}$ on the Artin-Schreier root stack $\mc{X}_m^{\nu}=\wp_{m}^{-1}((L,s,f)/X)$ which satisfies $\mc{N}^{\otimes p}\xrightarrow{\simeq}\pi^*L$. Moreover, $\mc{N}$ admits a $G$-linearization for the induced $G$-action on $\mc{X}_m^{\nu}$.
	\end{theorem}
	\begin{proof}
		Let $w : \mc{X}^{\nu}_m\longrightarrow \mc{X}_m$ denote the normalization. Choose an atlas $u:U\longrightarrow \mc{X}_m$\,, and consider the following Cartesian diagram:
		\begin{align*}
			\begin{gathered}
				\xymatrix{
					U^{\nu} \ar[rr]^{u'} \ar[d]_{w'} && \mc{X}_m^{\nu} \ar[d]^{w}\\
					U\ar[rr]^{u} && \mc{X}_m
				}
			\end{gathered}
		\end{align*}
		It was observed in Remark \ref{rem:special-atlas} that $u' : U^{\nu}\rightarrow\mc{X}^{\nu}_m$ is an atlas for the Artin-Schreier root stack $\mc{X}_m^{\nu}$. Consider the following  invertible sheaf on $\mc{X}^{\nu}_m$:
		$$\mc{N}:=w^*\mc{M}\,.$$ 
		Clearly, if $\mc{M}$ is given on $U$ by the invertible sheaf $M_u$\,, then $\mc{N}$ is given by the invertible sheaf $w'^*M_u$ on $U^{\nu}$. Also, the coarse moduli map $\pi : \mc{X}_m^{\nu}\longrightarrow X$ is equal to the composition $\pi = p\circ w$. The invertible sheaf $\mc{N}$ can therefore be thought of as the tautological sheaf on the Artin-Schreier root stack, which provides a $p$-th root for the line bundle $L$ on $X$, namely $\mc{N}^{\otimes p}\xrightarrow{\,\,\simeq\,\,}\pi^*L\,.$  Now, since the map $w: \mc{X}_m^{\nu}\rightarrow \mc{X}_m$ is $G$-equivariant by construction, it follows that the $G$-linearization on $\mc{M}$ will induce a $G$-linearization on $\mc{N}$ as well. This proves our claim.
	\end{proof}

\end{document}